\documentclass[11pt]{article}
\usepackage[utf8]{inputenc}
\usepackage[T1]{fontenc}
\usepackage{lmodern}
\usepackage{amsmath}
\usepackage{amsfonts}
\usepackage{amsthm}
\usepackage{amssymb}
\usepackage{dsfont}
\usepackage{blindtext}
\usepackage{titlesec}
\usepackage{mathtools}
\usepackage{epigraph}
\usepackage{enumitem}
\usepackage{geometry}
\usepackage[colorlinks=true, linkcolor=orange, citecolor=blue]{hyperref}
\usepackage{mathtools}
\usepackage{hyperref}
\usepackage{tabularx}
\usepackage{comment}
\usepackage{bigints}

\usepackage{tikz}
\usetikzlibrary{decorations.pathreplacing,calligraphy}
\usetikzlibrary{shapes,positioning,intersections,quotes,patterns,calc}
\usepackage{graphics}
\usepackage{graphicx}
\usetikzlibrary[topaths]
\usepackage{pgfplots}
\pgfplotsset{compat=1.18}

\usepackage{caption}
\usepackage{subcaption}

\newcommand{\ve}{\varepsilon}
\newcommand{\ct}{\mathcal{A}}
\newcommand{\dist}{\textnormal{dist}}
\newcommand{\ff}{F^{f,h}}
\newcommand{\tchi}{\tilde{\chi}}
\newcommand{\rj}{R_{j, \sigma}}
\newcommand{\tv}{\tilde{v}}
\newcommand{\tp}{t^\prime}
\newcommand{\M}{\mathcal{M}}
\newcommand{\rp}{R^\prime}
\newcommand{\tw}{\tilde{w}}
\newcommand{\thh}{\tilde{h}}
\newcommand{\sd}{s_\delta}
\newcommand{\rtt}{\tilde{r}}

\usepackage{mathrsfs} 
\usepackage{textcomp}
\usepackage[T1]{fontenc}
\usepackage{newunicodechar}
\numberwithin{equation}{section}

\theoremstyle{plain}
\newtheorem{thm}{Theorem}[section]
\newtheorem{lem}[thm]{Lemma}
\newtheorem{prop}[thm]{Proposition}
\newtheorem{cor}[thm]{Corollary}

\theoremstyle{definition}
\newtheorem{remark}[thm]{Remark}
\newtheorem{definition}{Definition}

\title{On the stability of an inverse problem for waves via the Boundary Control method}
\author{Spyridon Filippas\footnote{Department of Mathematics and Statistics, University of Helsinki, Helsinki, Finland, email: spyridon.filippas@helsinki.fi}, Lauri Oksanen\footnote{Department of Mathematics and Statistics, University of Helsinki, Helsinki, Finland, email: lauri.oksanen@helsinki.fi}}
\date{\today}

\makeatletter
\def\keywords{
    \vspace{1ex}
    \noindent
    \if@twocolumn
      \small{\bf  Keywords}\/---$\!$    \else
      \begin{center}\small\ {\bf Keywords}\end{center}\quotation\small
    \fi}
\def\endkeywords{\vspace{0.6em}\par\if@twocolumn\else\endquotation\fi
    \normalsize\rm}
\makeatother

\def\p{\partial}
\def\R{\mathbb R}
\DeclareMathOperator{\supp}{supp}
\newcommand{\norm}[3]{\left\Vert #1 \right\Vert_{#2}^{#3}}

\begin{document}

\maketitle

\begin{abstract}
 We establish a link between stability estimates for a hyperbolic inverse problem via the Boundary Control method and the blowup of a constant appearing in the contexts of optimal unique continuation and cost of approximate controllability. 
\end{abstract}


\tableofcontents

\section{Introduction}

\subsection{Background}

We consider the following wave equation:
\begin{equation}
\label{wave equation}
    \begin{cases}
    (\p^2_t-\Delta+q)u=f & \textnormal{in} \:(0,T) \times \R^n,\\
    u_{|t=0}=\p_tu_{|t=0}=0  & \textnormal{in} \: \R^n,
    \end{cases}
\end{equation}
where the potential $q \in C^\infty(\R^n;\R)$ is real valued, independent of the time variable $t$. Let $K \subset \R^n$ be a compact set on which one wants to recover $q$. We consider a (possibly small) observation set $\omega$ which is an open subset $\omega \subset \R^n$ having $C^1$ boundary and satisfying $\dist(K, \omega)>0$. We are interested in the following question: can one recover the values of the potential $q$ in the compact set $K$ in a stable way by using some sources $f$ in $\omega$ and by measuring the resulting waves in $\omega$?

In other words we consider for $f \in L^2((0,T)\times \omega)$ the \textit{source to solution map} $\Lambda_q$ associated to the potential $q$ given by:
\begin{align}
\label{def_of_lambda}
    \Lambda_q(f)=u_{|(0,T) \times \omega},
\end{align}
where $u$ is the solution of~\eqref{wave equation}. Here we write $f \in L^2((0,T)\times \omega)$ for $f\in L^2$ with $\supp(f)\subset (0,T)\times \omega$. We want to show a stability estimate of the form
\begin{align*}
 \norm{q_1-q_2}{*}{} \leq \theta\left( \norm{\Lambda_{q_1}- \Lambda_{q_2}}{**}{}\right),
\end{align*}
for appropriate norms and $\theta$ a modulus of continuity. We will write in the sequel
$$
\Lambda_{q_j}=\Lambda_j, \quad j \in \{1,2\}.
$$

A closely related problem is the one of determining a metric $g$ or a potential $q$ from the knowledge of the \textit{Dirichlet-to-Neumann map} (DN map in short). From a physical point view, this corresponds to determining the properties of a medium by probing it at the~\textit{boundary} and using as data the responses of the medium measured at the boundary. Mathematically speaking, we consider a compact Riemannian manifold $(\M, g)$ and $u$ satisfying
\begin{equation*}
\label{wave equation manifold}
    \begin{cases}
    (\p^2_t-\Delta_g+q)u=f & \textnormal{in} \:(0,T) \times \M,\\
    u_{|t=0}=\p_tu_{|t=0}=0  & \textnormal{in} \: \M, \\
    u=f, & \textnormal{in} \: (0,T) \times \partial M.
    \end{cases}
\end{equation*}
We define then the $DN$ map in a similar fashion by $\Lambda^{g,q}_{DN}(f)=\partial_\nu u_{|(0,T)\times \M}$ and one can ask if the knowledge of the DN map $\Lambda^{g,q}_{DN}$ allows to uniquely determine the metric or the potential, and if yes, if it is possible to give a stability estimate. Concerning the uniqueness question there are essentially two approaches:
\begin{itemize}
    \item The first is based on the Boundary Control method, originating from~\cite{Belishev:87}, and was used to prove that indeed the DN map determines the metric $g$ or the potential $q$ (see e.g the book~\cite{bookMatti}). This result holds without any restriction on the Riemannian metric $g$. In this approach the property of~\textit{unique continuation} for the wave equation plays a key role in the proof.

    \item The second approach is based on the construction of geometric optics solutions. This allows to reduce the problem to the one of studying the injectivity properties of the geodesic ray transform. Since the geodesic ray transform is known to be invertible under certain geometric assumptions (see e.g the book~\cite{paternain2023geometric}) this method gives positive results when $(\M,g)$ satisfy such assumptions. Such results essentially impose that the metric $g$ should be simple\footnote{We recall that these are simply connected $(\M,g)$ with no conjugate points and strictly convex boundary.} or satisfy some milder (but closely related) conditions. 
\end{itemize}

Now the problem of proving~\textit{stability estimates} without any particular assumption on the metric $g$ was not studied until recently. Indeed, to the best of our knowledge the only works in this direction are the ones by Bosi-Kurylev-Lassas~\cite{bosi2017reconstruction} and Burago-Ivanov-Lassas-Lu~\cite{burago2024quantitativestabilitygelfandsinverse}. In these references the authors prove stability estimates of $\log \log$ type in the case where there is no potential. 

In the case where one imposes additional conditions on $g$, so that the geometric optics approach can be used, it seems that the bibliography is much richer. Indeed, when $g$ is the Euclidean metric a Hölder stability estimate was proved by Sun~\cite{sun1990continuous} and Alessandrini-Sun-Sylvester~\cite{alessandrini1990stability} for the determination of the potential. In \cite{stefanov1998stability} Stefanov and Uhlmann proved Hölder stability for the recovery of metrics which are close enough to the Euclidean metric, and subsequently generalized this for generic simple metrics in~\cite{stefanov2005stable}. In the last two works there is no potential involved. In~\cite{MouradDos} Bellassoued and Dos Santos Ferreira consider a fixed simple metric and prove a Hölder stability estimate for the recovery of the potential. We mention as well the work of Arnaiz-Giullarmou~\cite{arnaiz2022stability} where they prove a Hölder stability estimate for the recovery of the potential under a condition on the trapped set of the manifold, which relaxes the simplicity assumption.

As already mentioned, the proofs of these Hölder stability results are based on the geometric optics approach. However, in the problem that we consider in the present paper, in which the DN map is replaced by the source to solution map, this approach cannot be used. This stems roughly speaking from the fact that the geometric optics approach only works when the probing/observation region surrounds in some sense the domain of interest. In our setting this means that, for each point $x \in K$, all the rays (that is, half-lines) from $x$ intersect $\omega$.
In the case of general $K$ and $\omega$, the Boundary Control method is the only known method for proving uniqueness. We consider stability in the general setting, and our proof is based on this method. 

The Boundary Control method heavily relies on unique continuation/approximate controllability properties for the wave equation which turn out to be quite subtle (see Section 2 for more details). It seems natural that a quantitative result based on the Boundary Control should be related to some quantitative result for the unique continuation of the wave operator. In this paper we prove an abstract result which establishes exactly the link between these two: stability for the recovery of a potential from the knowledge of $\Lambda$ on one side and the cost of the approximate control constant $\ct$ related to unique continuation for waves on the other. This is the content of Theorem~\ref{main thm}.

\subsection{Main result and applications}

We denote by $$\mathcal{L}(K, \omega)= \sup_{x \in K}\textnormal{dist}(x,\omega),$$ the largest distance of the subset $\omega$ to a point of $K$. We have the following.
\begin{thm}
\label{main thm}
    Assume that $T>2\mathcal{L}(K, \omega)$ and the potentials $q_j$ satisfy $\norm{q_j}{C^m(\R^n)}{}\leq M, j \in \{1,2\}$ for some $M>0$ and $m >n$. Then given small $\ve>0$ one has the following stability estimate.

If $$
\norm{\Lambda_1-\Lambda_2}{L^2\to L^2}{} \leq \left( \frac{\ve}{\ct(\ve^\ell)}\right)^\ell,
$$
where the operator norm is from $L^2((0,T) \times \omega)$ to itself, then 
$$\norm{q_2-q_1}{L^2(K)}{}\leq C \ve,$$
where:
\begin{itemize}
    \item $C>0$ is a constant depending only on $T,M,K,\omega,n$,

    \item $\ell=\ell(n)>1$ depends only on the dimension $n$,

    \item $\ct=\ct(\ve)$ is the cost of approximate control constant of Theorem~\ref{cor approx control}.
\end{itemize}
\end{thm}

\bigskip

In light of Theorem~\ref{main thm} it becomes interesting to study the behavior of the quantity $\ct(\ve)$ as $\ve$ approaches $0$. Partly motivated by this, in the companion paper~\cite{FO25explicit} we study this problem in the particular case where the control region is a ball of radius $r$. This allows to obtain a stability result when the observation takes place on a \textit{half-space}. 

Let us denote by $H$ the half-space $\{x_1>0\}$ in $\R^n_x$. Let $q_{j} \in C^\infty (\R^n; \R), j\in\{1,2\}$ and $K$ a compact set of $\{x_1<0\}$. Consider a source $F$ supported in $(0,T)\times H$ and the source to solution map $\Lambda_{q_j}(F)=u_{|(0,T)\times H}$. We have the following.

\begin{thm}
\label{thm_inverse_half}
Assume that the potentials $q_j$ satisfy $\norm{q_j}{C^m(\R^n)}{}\leq M, j \in \{1,2\}$ for some $M>0$ and $m >n$. Then for $T>2\mathcal{L}(K, H)$ there exist $C>0$ and $\alpha \in (0,1)$ depending only on $T,K,M$ with 
\begin{align*}
    \norm{q_2-q_1}{L^2(K)}{}\leq \frac{C}{\left(\log |\log \norm{\Lambda_1-\Lambda_2}{L^2\to L^2}{} |\right)^\alpha}.
\end{align*}
\end{thm}

Theorem~\ref{thm_inverse_half} is a consequence of Theorem~\ref{main thm} and \cite[Theorem 1.1]{FO25explicit}. We give the details in Section~\ref{section_end_proofs}.

\section{Finite speed of propagation and unique continuation}
\label{sec_finite and uc}

We recall here some classical facts on the wave equation that will be crucial in the sequel. We write $u^f_j,$ $j\in \{1,2\}$, for the solution of~\eqref{wave equation} associated to the potential $q_j$ with source term $f$. We also denote $u(t):=u(t, \cdot)$.

\begin{definition}
    For $s>0$ and $\omega \subset \R^n$ an open set we define the subset
    $$
    M(\omega,s):=\{x \in \R^n: \textnormal{dist}(x, \omega) \leq s\}
    $$
The set $M(\omega,s)$ is called the \textit{domain of influence} of $\omega$ at time $s$.
\end{definition}
We then have:
\begin{thm}[Finite speed of propagation for waves]
\label{finite speed}
Let $\omega \subset \R^n$ and $u^f$ be a solution of~\eqref{wave equation} with a source $f \in H^1((0,T) \times \R^n)$ satisfying $\supp(f) \subset (0,T] \times \omega$ and a potential $q \in L^\infty(\R^n)$. Then $\supp (u^f(s ))\subset M(\omega,s)$.
\end{thm}
We refer to~\cite[Theorem 2.47]{bookMatti} for a proof of this classical result.

\bigskip

The property of \textit{unique continuation} can be seen as a reciprocal of the finite speed of propagation and it consists in asking the following question: if a solution of the wave equation vanishes identically on some open set $\omega$ for time $T>0$, does it vanish everywhere? It is interesting to note that in striking contrast with the property of finite speed of propagation which is based on energy estimates, with a proof that is essentially the same in all dimensions, the property of unique continuation turns out to be much more intricate. A seminal result for the wave equation was achieved in~\cite{Robbiano:91} where the property of unique continuation from $(-T,T) \times \omega$ is proved as soon as $\omega \neq \emptyset$ and $T>0$ is sufficiently big. The minimal time of observation $T$ was subsequently improved in~\cite{Hormander:92}. However, the Boundary Control method relies crucially in the interplay between finite speed of propagation and unique continuation and one really needs a unique continuation result in~\textit{optimal time}. This was achieved by Tataru in~\cite{Tataru:95} and eventually led to a very general unique continuation result: the Tataru-Robbiano-Zuily-Hörmander~\cite{Tataru:95, tataru1999unique, RZ:98,Hor:97} Theorem. We refer to~\cite{LL:23notes} for an introduction to this result in the particular case of the wave equation. Here, we shall use the following:

\begin{thm}[Unique continuation]
\label{thm unique continuation}
   Let $\omega \subset \R^n$ be a non empty open set. Let $s>0$ and $u \in H^1_{\textnormal{loc}}((-s,s)\times \R^n)$ satisfying $(\p^2_t-\Delta+q)u=0$ on $(-s,s)\times \R^n$. Assume that $u_{|(-s,s) \times \omega}=0$. Then $u(0)_{|M(\omega,s)}=0$.
\end{thm}
We remark that the domain where unique continuation is achieved is precisely the maximal domain so that finite speed of propagation is not violated and in this sense the result above is optimal, see e.g~\cite{russell1971boundarya,russell1971boundaryb}. 

Theorem~\ref{thm unique continuation} is in fact an injectivity result which in turn yields by duality a density result called \textit{approximate controllability}. We shall use the following formulation, which introduces an abstract constant $\ct=\ct(\ve)$ corresponding to the~\textit{cost of the control}.

\begin{thm}[Approximate controllability]
\label{cor approx control}
  Let $\omega \subset \R^n$ be a non empty open set and $s \in (0,T]$. Then for all $v \in H^1(M(\omega,s))$ and any precision $\varepsilon>0$ there exists $f_{\varepsilon} \in C^\infty_0((0,s)\times \omega)$ such that 
  $$
  \norm{u^{f_\ve}(s)-v}{L^2(M(\omega,s))}{}\leq \varepsilon \norm{v}{H^1(M(\omega,s))}{},
  $$
  with $u^{f_\ve}$ solution of~\eqref{wave equation}. We denote by $\ct(\ve)$ the cost of the approximate control satisfying 
  $$
  \norm{f_\ve}{L^2(M(\omega,s))}{}\leq \ct(\ve) \norm{v}{H^1(M(\omega,s))}{}.
  $$
\end{thm}

Corollary~\ref{cor approx control} states that one can use the source $f$ (the control) in order to drive the solution from $0$ at time $t=0$ to $\varepsilon$ close to $v $ at time $t=s$. Notice that the support of $f$ and the finite speed of propagation of Theorem~\ref{finite speed} guarantee that $u^f(s)$ is indeed supported in $M(\omega,s)$. 

By duality, as proved in~\cite{Robbiano:95}, a quantitative unique continuation result allows to obtain a bound on the cost of the approximate control for the wave equation. This was achieved with an optimal stability in~\cite{Laurent_2018} and in an almost optimal way but with more explicit dependency on different geometric parameters in~\cite{BKL:16,bosi2018stability}. However, in these references the unique continuation (or equivalently control) is achieved in a slightly smaller domain than the optimal one. We refer to the introduction of the companion paper~\cite{FO25explicit} for more details.

\begin{remark}
\label{remark on time of control}
In classical control theory one is also interested in controlling the derivative $\p_t u_{|t=s}$ of the solution of~\eqref{wave equation} at time $t=s$. However, this is not needed for the Boundary Control method and it is in fact of crucial importance that we only need to control $u(s)$. Indeed, in classical control theory, approximate controllability at time $s$ from an open set $\omega$ is~\textit{equivalent} to unique continuation in time $s$ from $\omega$, see for instance~\cite[Section 1.1.2]{LL:23notes}. Theorem~\ref{thm unique continuation} asserts that the minimal time for unique continuation to hold from the set $\omega$ in $M(\omega,s)$ is $2s$, whereas in Corollary~\ref{cor approx control} we only need time $s$. This is due precisely to the fact that we only drive the solution $u$ and not its time derivative. Let us illustrate why. 

We write $\M=M(\omega,s)$ and define the final value map by
\begin{align*}
    F_s: L^2((0,s); H^{-1}(\omega)) & \to L^2(\M) \\
    f & \mapsto u^f(s).
\end{align*}
Let now $h \in L^2(\M)$ and consider $w$ the solution of
\begin{equation}
\label{system for adjoint}
 \begin{cases}
    (\p^2_t-\Delta+q)w=0 & \textnormal{in} \:(0,s) \times \M,\\
    w_{|t=s}=0, & \textnormal{in} \: \M\\
    \p_t w_{|t=s}=-h  & \textnormal{in} \: \M,
    \end{cases}   
\end{equation}
with Dirichlet boundary conditions. Integrating by parts in space and time the equality $$\int \int_{(0,s)\times \M} (\p^2_t-\Delta+q) uw dtdx = \int \int_{(0,s)\times \M} fw dt dx, $$ yields
$$
(u(s),h)_{L^2(\M)}=(f,u_{|(0,s)\times \omega})_{L^2;H^{-1},L^2;H^1}.
$$
In other words the adjoint $F^*_s$ of $F_s$ can be identified with the map
\begin{align*}
    F^*_s: L^2(\M) & \to L^2((0,s); H^{1}(\omega))   \\
    h & \mapsto w_{|(0,s)\times \omega},
\end{align*}
with $w$ solution of~\eqref{system for adjoint}. Then the approximate controllability property, that is $\overline{\textnormal{range}(F_s)}=L^2(\M)$ is equivalent to $\textnormal{ker}(F^*_s)=\{0\}$. Let $h \in \textnormal{ker}(F^*_s) $. We need to show the unique continuation property 
$$
\begin{cases} 
 (\p^2_t-\Delta+q)w=0  \:   \textnormal{in } (0,s) \times \M \\ w=0 \:   \textnormal{in}  \:  (0,s) \times \omega 
\end{cases}
\Longrightarrow w \equiv 0.
$$
Notice that the time $s$ is not sufficient to apply Theorem~\ref{thm unique continuation}. The key point is that since $w$ satisfies $w(s)=0$ (condition that we were able to impose exploiting the fact that we do not need to control the time derivative $\p_t u$) we can consider the extension given by $\tilde{w}(t)=-w(2s-t)$ for $t \in [s,2s]$. Then $\tilde{w} \in H^1((0,2s) \times \M)$ and satisfies 
$$
\begin{cases} 
 (\p^2_t-\Delta+q) \tilde{w}=0  \:   \textnormal{in } (0,2s) \times \M \\ \tilde{w}=0 \:   \textnormal{in}  \:  (0,2s) \times \omega,
\end{cases}
$$
and we can now use unique continuation from $\omega$ to $\M= M(\omega,s)$ in minimal time $2s$ to conclude.
\end{remark}

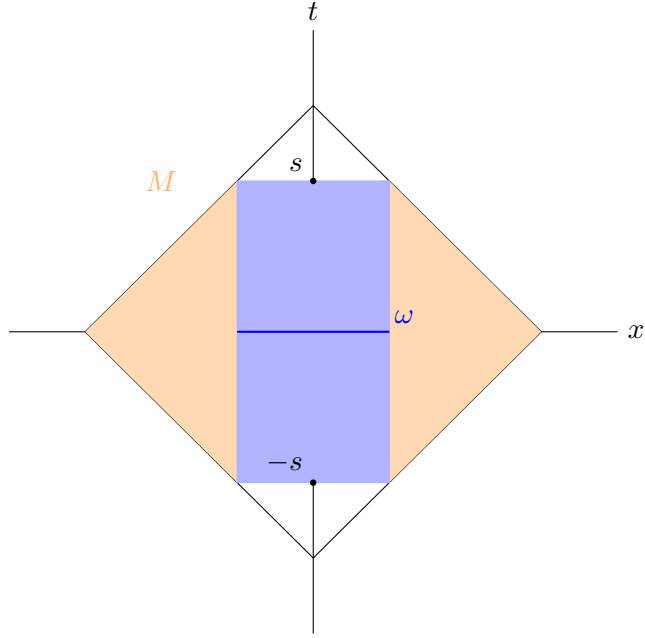
\begin{figure}
	\centering
	
	\begin{tikzpicture}
\draw (0,-4) -- (0,4) node[above] {$t$} ;
\draw (-4,0) -- (4,0) node[right] {$x$};

\draw (-3,0) -- (0,3) ;
\draw (0,3) -- (3,0) ;
\draw (3,0) -- (0,-3) ;
\draw (0,-3) -- (-3,0) ;


\fill[fill=orange!30,semitransparent] (-3+0,0) -- (-1,2-0)-- (-1,-2+0)--(-3+0,0) ;

\fill[fill=orange!30,semitransparent] (3+0,0) -- (1,2-0)-- (1,-2+0)-- (3+0,0) ;

\filldraw[blue!30, semitransparent]  (-1,-2) rectangle (1,2) ;
\draw [blue, thick] (-1,0) -- (1,0);

\node at (1.2,0.2) [blue] {$\omega$};

\node at (-2,2) [orange!60] {$M$};

\filldraw[black] (0,2) circle (1pt) node[anchor=south east]{$s$};
\filldraw[black] (0,-2) circle (1pt) node[anchor=south east]{$-s$};

\end{tikzpicture}
\caption{A solution of the wave equation that vanishes on $(-s,s)\times \omega$ has to vanish in the set $M$. This is the largest diamond above. The domain $M(\omega,s)$  is $M\cap \{t=0\}$.}
\label{cone and delta}
\end{figure}

\section{Proof of the stability estimate}

\subsection{Stable recovery of the inner products}

We recall that the source to solution operator $\Lambda_q$ is continuous,
$$
\Lambda_q: H^k_0((0,T) \times \omega) \mapsto H^{k+1}((0,T) \times \omega).
$$
In the sequel we will however see $\Lambda_q$ as an operator in $L^2$. Its adjoint is given by the following lemma:

\begin{lem}
    \label{adjoint of lambda}
Let us denote by $Ru(t):=u(T-t)$ the time reversal operator on $(0,T)$. Then one has $\Lambda^*_q=R\Lambda_{q}R$. 
\end{lem}

\begin{proof}
Let $f,h \in L^2((0,T)\times \omega)$ and consider $v$ the solution of 
$$
 \begin{cases}
    (\p^2_t-\Delta+q)v=h & \textnormal{in} \:(0,T) \times \R^n,\\
    v_{|t=T}=\p_tv_{|t=T}=0  & \textnormal{in} \: \R^n.
    \end{cases}
$$
Using the support properties of $f$ and $h$, the initial conditions on $u$ and $v$ and integrating by parts in time and space gives:
\begin{align*}
    \left(\Lambda_q f, h\right)_{L^2((0,T)\times \omega)}&=\int_{(0,T)\times \R^n} (\Lambda_q f) \Bar{h}dtdx=\int_{(0,T)\times \R^n} u  (\p^2_t-\Delta+q)\Bar{v} dtdx \\
    &=\int_{(0,T)\times \R^n} (\p^2_t-\Delta+q) u  \Bar{v} dtdx=\int_{(0,T)\times \R^n} f \Bar{v} dtdx \\
    &= \left(f, v_{|(0,T)\times \omega)}\right)_{L^2((0,T)\times \omega)}.
\end{align*}
The result follows by noticing that $v_{|(0,T)\times \omega)}=(R\Lambda_{q}R) (h).$
\end{proof}

The following Blagoveščenskii type identity is a crucial ingredient of the Boundary Control method and goes back to~\cite{blagoveshchenskii1967inverse}. We use a formulation similar to~\cite{Bingham2008IterativeTC}.

\begin{lem}
\label{the blagovetc identity}
    Let $T>0$, $f,h \in L^2((0,T)\times \omega)$. Then
$$
\left(u^f(T/2),u^h(T/2) \right)_{L^2(\R^n)}= \left(f, K(\Lambda)h\right)_{L^2((0,T/2)\times \omega)},
$$
where the operator $K=K(\Lambda)$ is given by
\begin{equation}
    \label{expression of K}
    K=R\Lambda R J+J\Lambda,
\end{equation}
with $J$ given by
$$
Jf(s):=\frac{1}{2}\int_{s}^{T-s}f(s^\prime)ds^\prime, \quad f\in L^2(0,T).
$$
\end{lem}

\begin{proof}
    We define $W(t,s):=\left( u^f(t),u^h(s) \right)_{L^2(\R^n)}$ and calculate
    \begin{align*}
        (\p^2_t-\p^2_s)W(t,s)&= \left( \p^2_t u^f(t),u^h(s) \right)_{L^2(\R^n)}-\left( u^f(t), \p^2_s u^h(s) \right)_{L^2(\R^n)} \\
        &=\left( \Delta u^f(t)+f(t)-qu^f(t),u^h(s) \right)_{L^2(\R^n)}-\left( u^f(t), \Delta u^h(s) +h(s)-q u^h(s)\right)_{L^2(\R^n)}.
    \end{align*}
Integrating by parts, using the support of $f$ and recalling the definition of $\Lambda$ the first term becomes:
\begin{align*}
    \left( \Delta u^f(t)+f(t)-qu^f(t),u^h(s) \right)_{L^2(\R^n)}&=-\left(\nabla u^f(t),\nabla u^h(s)\right)_{L^2(\R^n)}+\left(f(t), (\Lambda h)(s) \right)_{L^2( \omega)} \\
    & \hspace{4mm}-\left(q u^f(t),u^h(s) \right)_{L^2(\R^n)}.
    \end{align*}
Similarly, we find for the second term (we use here that $q$ is real valued):
\begin{align*}
    \left( u^f(t), \Delta u^h(s) +h(s)-q u^h(s)\right)_{L^2(\R^n)}&=-\left(\nabla u^f(t),\nabla u^h(s)\right)_{L^2(\R^n)}+\left( (\Lambda f)(t),  h(s) \right)_{L^2( \omega)}\\
    & \hspace{4mm}-\left(q u^f(t),u^h(s) \right)_{L^2(\R^n)}
\end{align*}
Putting all of the above together yields
\begin{align*}
 (\p^2_t-\p^2_s)W(t,s)&=F(t,s),
\end{align*}
with
\begin{equation}
    \label{def of F}
    F(t,s):=\left(f(t), (\Lambda h)(s) \right)_{L^2(\omega)}-\left( (\Lambda f)(t),  h(s) \right)_{L^2(\omega)}.
\end{equation}
We moreover have 
$$
W_{|t=0}=\p_t {W}_{|t=0}=W_{|s=0}=\p_s {W}_{|s=0}=0.
$$
Then $W$ solves the following $1D$ wave equation:
$$
 \begin{cases}
    (\p^2_t-\p^2_s)W(t,s)=F(t,s) & \textnormal{in} \:(0,T) \times (0,T) ,\\
    W_{|s=0}(t)=\p_s {W}_{|s=0}=0  & \textnormal{in} \: (0,T), \\
    W_{|t=0}(s)=\p_t {W}_{|t=0}=0 & \textnormal{in} \: (0,T).
    \end{cases}
$$
We can solve the above system explicitly and obtain the expression
\begin{equation}
\label{expression for W}
    W(t,s)=\frac{1}{2}\int_0^t \int_{s-r}^{s+r} F(t-r,y)dydr, \quad (t,s)\in (0,T/2) \times (0,T/2).
\end{equation}
This gives
\begin{align*}
    W(T/2,T/2)&=-\frac{1}{2}\int_{0}^{T/2}\int_{s}^{T-s}F(s,y)dy ds=-\int_0^{T/2}(JF(s,\cdot))(s) ds \\
    &=-\left(f, J\Lambda h  \right)_{L^2((0,T/2)\times \omega)}+\left(\Lambda f, J h  \right)_{L^2((0,T/2)\times \omega)} \\
    &=\left(f, (-J\Lambda+\Lambda^* J) h  \right)_{L^2((0,T/2)\times \omega)},
\end{align*}
and the result follows by recalling the expression of $\Lambda^*$ in Lemma~\ref{adjoint of lambda}.
\end{proof}

Lemma~\ref{the blagovetc identity} allows to obtain a linear stability estimate for the difference of inner products associated to two solutions $u_1,u_2$:
\begin{prop}
\label{prop stable recovery of inner products}
    Let $0<t, \tp<T/2$, $f,h \in L^2((0,T)\times \omega)$ and consider the solutions $u_1$ and $u_2$ associated to the potentials $q_1,q_2$. Then we have the estimate
    \begin{equation}
        \label{stability estimate for inner products}
        \left| \left(u_1^f(t),u_1^h(\tp) \right)_{L^2(\R^n)}- \left(u_2^f(t),u_2^h(\tp) \right)_{L^2(\R^n)}  \right| \leq C \norm{f}{L^2((0,T)\times \omega)}{}\norm{h}{L^2((0,T)\times \omega)}{} \norm{\Lambda_1-\Lambda_2}{L^2\to L^2}{},
    \end{equation}
with a constant $C>0$ depending only on $T$.
\end{prop}

\begin{proof}
Let us consider $s:=T/2-t>0$ and $s^\prime:=T/2-t^\prime>0$. We denote by $\tau_s$ the translation by $s$ given by $\tau_s(g)(t):=g(t-s)$ and observe that $u_j^f(t)=v^{\tau_s f}_j(T/2)$ with $v_j:=v^{\tau_s f}_j$ satisfying 
$$
\begin{cases}
    (\p^2_t-\Delta+q)v_j=\tau_s f & \textnormal{in} \:(s,T+s) \times \R^n,\\
    v_{j|t=s}=\p_tv_{j|t=s}=0  & \textnormal{in} \: \R^n.
\end{cases}
$$
We define similarly $u_j^h(t^\prime)=v^{\tau_{s^{\prime}} h}_j(T/2)$ and use the calculations of Lemma~\ref{the blagovetc identity} to find:
    \begin{align*}
    &\left| \left(u_1^f(t),u_1^h(\tp) \right)_{L^2(\R^n)}- \left(u_2^f(t),u_2^h(\tp) \right)_{L^2(\R^n)}  \right|\\
    &=\left| \left(v^{\tau_s f}_1(T/2),v^{\tau_{s^{\prime}} h}_1(T/2) \right)_{L^2(\R^n)}- \left(v^{\tau_s f}_2(T/2),v^{\tau_{s^{\prime}} h}_2(T/2) \right)_{L^2(\R^n)}  \right|\\
    &= \left| \left(\tau_s f, K(\Lambda_1-\Lambda_2)\tau_{s^{\prime}}h\right)_{L^2((s,T/2+s)\times \omega)}  \right| \\
        &=\left| \left(f, K(\Lambda_1-\Lambda_2)h\right)_{L^2((0,T/2)\times \omega)}  \right|\leq \norm{f}{L^2}{} \norm{K(\Lambda_1-\Lambda_2)h}{L^2((0,T/2)\times \omega)}{}.
    \end{align*}
We observe now that $\norm{R}{L^2\to L^2}{}=1$ and by the Cauchy-Schwarz inequality one has $\norm{Jf}{L^2(0,T/2)}{}\leq C \norm{f}{L^2(0,T)}{}$. Recalling the expression of $K$ in~\eqref{expression of K} yields the estimate.
\end{proof}

\subsection{From inner products to domains of influence}
The crucial step now is to introduce characteristic functions of domains of influence in the estimate~\eqref{stability estimate for inner products}. This is achieved by using approximate controllability. This is the least stable step of the whole proof.

Let us consider $0<s<t<T/2$, $f\in L^2((0,T)\times \omega)$ and $ g \in L^2((t-s,t)\times \omega) $. We calculate:
\begin{align*}
    \norm{u^g(t)-u^f(t)}{L^2}{2}&=\norm{u^g(t)-\mathds{1}_{M(\omega,s)}u^f(t)}{L^2}{2}+\norm{(1-\mathds{1}_{M(\omega,s)})u^f(t)}{L^2}{2}\\
    &\hspace{4mm}-2 \left(u^g(t)-\mathds{1}_{M(\omega,s)}u^f(t),(1-\mathds{1}_{M(\omega,s)})u^f(t)\right)_{L^2}.
\end{align*}
Notice that by finite speed of propagation one has $\supp(u^g(t)) \subset M(\omega,s)$ and consequently the inner product in the last term above vanishes. This gives the identity:
\begin{equation}
    \label{identity sum of squares}
     \norm{u^g(t)-u^f(t)}{L^2}{2}=\norm{u^g(t)-\mathds{1}_{M(\omega,s)}u^f(t)}{L^2}{2}+\norm{\left(1-\mathds{1}_{M(\omega,s)}\right )u^f(t)}{L^2}{2}.
\end{equation}
To alleviate notation we define the following quantities:
\begin{align}
\label{def of As}
    A_j(g)=A_j(g)(t)&= \norm{u_j^g(t)-u_j^f(t)}{L^2}{2}, \\
    \label{def of Bs}
    B_j(g)=B_j(g)(t,s)&= \norm{u_j^g(t)-\mathds{1}_{M(\omega,s)}u_j^f(t)}{L^2}{2}, \\
    \label{def of Gammas}
    \Gamma_j=\Gamma_j(t,s)&=\norm{\left(1-\mathds{1}_{M(\omega,s)}\right )u_j^f(t)}{L^2}{2},
\end{align}
where $u_j, j \in \{1,2\}$ are the solutions associated to potentials $q_1,q_2$ and $f\in L^2((0,T)\times \omega)$ is fixed and omitted from the notation for better readability. In particular, identity~\eqref{identity sum of squares} becomes with this notation
\begin{equation}
    \label{identity with A,B,Gamma}
     A_j(g)=B_j(g)+ \Gamma_j, \quad j \in \{1,2\}, \: g \in  L^2((t-s,t) \times \omega).
\end{equation}

Let us now illustrate how one exploits approximate controllability in the qualitative case. Suppose that we have $\Lambda_1=\Lambda_2$. Then Proposition~\ref{prop stable recovery of inner products} implies $A_1(g)=A_2(g)$ and~\eqref{identity with A,B,Gamma} gives
\begin{equation}
    \label{identity with B,Gamma}
    B_1(g)+ \Gamma_1=B_2(g)+ \Gamma_2.
\end{equation}
Since $B_j(g)+\Gamma_j\geq \Gamma_j$ the approximate controllability result of Theorem~\ref{cor approx control} yields 
$$ \displaystyle  \inf_{g \in L^2((t-s,t) \times \omega)} \left(B_j(g)+\Gamma_j \right) =\Gamma_j.
$$ 
Taking the infimum over $g$ on both sides of~\eqref{identity with B,Gamma} we find that $\Gamma_1=\Gamma_2$ which in turn implies once again thanks to~\eqref{identity with B,Gamma} that $B_1(g)=B_2(g)$. This equality essentially allows to introduce the quantity $\mathds{1}_{M(\omega,s)}u_j^f(t)$ in the inner products in~\eqref{stability estimate for inner products}. For the stability problem we need to provide an estimate for the difference $\left| B_1(g)-B_2(g) \right|.$ 

An important remark is that the infimum of the quantity $B_j+\Gamma_j$ is, in general, not attained. Indeed, that would imply the existence of an \textit{exact} control $g$ and it is known since the seminal work of ~\cite{BLR:92} that this is essentially equivalent to the Geometric Control Condition for $\omega$ (see as well~\cite{BG:97}). In the geometric context under consideration one would need an exact control from $\omega \subset M(\omega,s)$ in time $s$. The Geometric Control Condition in this case asks that line segment $\Gamma$ of length $s$ intersecting $M(\omega,s)$ satisfies $\omega \cap \Gamma \neq \emptyset.$ This is an extremely restrictive condition on $\omega$.

We consider then a small regularization parameter $\alpha>0$ and study a minimization problem for $A(g)+\alpha \norm{g}{L^2}{2}=\norm{u^g(t)-u^f(t)}{L^2}{2}+\alpha \norm{g}{L^2}{2}$ (we drop the index $j$ to alleviate notation). We have the following lemma:

\begin{lem}
    \label{minimazation prblem properties}
    Let $\alpha \in (0,1)$ and define $$A_\alpha(g):= A(g)+\alpha \norm{g}{L^2((t-s,t) \times \omega)}{2}=\norm{u^g(t)-u^f(t)}{L^2}{2}+\alpha \norm{g}{L^2((t-s,t) \times \omega)}{2}.$$ Then the minimization problem
    $$
   \min_{g \in L^2((t-s,t) \times \omega)} A_\alpha(g),
    $$
    admits a unique solution $g=g_\alpha$. The minimizer $g_\alpha$ is the unique solution of the equation
    $$
    (K(\Lambda)+\alpha)g_\alpha=K(\Lambda)f,
    $$
and satisfies the estimate
\begin{equation}
    \label{estim for galpha}
    \norm{g_\alpha}{L^2((t-s)\times \omega)}{}\leq \frac{C}{\alpha} \norm{f}{L^2((0,T)\times \omega)}{},
\end{equation}
for a constant $C$ depending only on $T$.
\end{lem}
\begin{proof}
The proof is the same as in~\cite[Lemma 2]{Bingham2008IterativeTC} and we omit it. Estimate~\eqref{estim for galpha} comes from the fact that $K(\Lambda)$ is positive and self-adjoint (which is a consequence of the relation of Lemma~\ref{the blagovetc identity}). This implies
$$
\norm{(K(\Lambda)+\alpha)^{-1}}{L^2 \to L^2}{}= \frac{1}{\textnormal{dist}(\sigma(K(\Lambda)),-\alpha)}\leq \frac{1}{\alpha},
$$
and hence~\eqref{estim for galpha}.
\end{proof}

In the qualitative argument illustrated before, we used the approximate controllability result of Theorem~\ref{cor approx control} which is based on unique continuation from the set $(-s,s)\times \omega$ to $M(\omega,s)$. This is the \textit{optimal} domain where unique continuation holds and one needs to quantify this property.

We then have the following key proposition:
\begin{prop}
    \label{prop difference of Bs}
For all $\varepsilon>0$ there exist $g_{j}=g_{j}(\varepsilon) \in L^2((t-s,t)\times \omega),$ $j \in \{1,2\}$ such that for all $\alpha \in (0,1)$ one has the estimate
\begin{align}
    \label{estim difference of Bs}
    &\left| B_1(g)-B_2(g) \right|  \nonumber \\
    &\leq C \bigg(\varepsilon^2 \norm{f}{L^2}{2}+\alpha( \norm{g_{1}}{L^2((t-s,t) \times \omega)}{2}+\norm{g_{2}}{L^2((t-s,t) \times \omega)}{2})  \nonumber \\
    &\hspace{4mm} +\norm{\Lambda_1-\Lambda_2}{L^2 \to L^2}{}\left(\norm{g_{1,\alpha}}{L^2}{2}+\norm{g_{2,\alpha}}{L^2}{2}+\norm{g}{L^2}{2}+\norm{f}{L^2}{2}\right) \bigg).
\end{align}
Moreover, we have the bound
\begin{equation}
    \label{bounds for control cost}
\norm{g_{j}}{L^2}{}\leq \ct(\ve) \norm{f}{L^2}{},
\end{equation}
with $C$ depending on $T,M,K,\omega,n$ and $\ct(\ve
)$ the cost of approximate control constant of Theorem~\ref{cor approx control}.
\end{prop}
We recall that $g_{j,\alpha}$ denote the minimizers of the quantities $A_{j,\alpha}$.

\begin{proof}
Recalling the definition of $A_j$ in~\eqref{def of As} and using Lemma~\ref{prop stable recovery of inner products} yields
\begin{equation}
    \label{estim difference ofAs}
    \left| A_1(g)-A_2(g) \right| \leq C \norm{\Lambda_1-\Lambda_2}{L^2 \to L^2}{}\left( \norm{f}{L^2}{2}+ \norm{g}{L^2}{2} \right).
\end{equation}
We now want to show that the minima of $A_{j,\alpha}$ are close to each other when the source to solution maps $\Lambda_j$ are close to each other. We have:
\begin{align*}
 A_{2,\alpha}(g)= A_{1,\alpha}(g)+A_2(g)-A_1(g)\geq A_1(g_{1,\alpha})- C \norm{\Lambda_1-\Lambda_2}{L^2 \to L^2}{}\left( \norm{f}{L^2}{2}+ \norm{g}{L^2}{2} \right),
\end{align*}
thanks to~\eqref{estim difference ofAs}. Choosing $g=g_{2,\alpha}$ we get
$$
A_{2,\alpha}(g_{2,\alpha}) \geq A_1(g_{1,\alpha})- C \norm{\Lambda_1-\Lambda_2}{L^2 \to L^2}{}\left( \norm{f}{L^2}{2}+ \norm{g_{2,\alpha}}{L^2}{2} \right),
$$
and we similarly find
$$
A_{1,\alpha}(g_{1,\alpha}) \geq A_2(g_{2,\alpha})- C \norm{\Lambda_1-\Lambda_2}{L^2 \to L^2}{}\left( \norm{f}{L^2}{2}+ \norm{g_{1,\alpha}}{L^2}{2} \right).
$$
These two imply
\begin{equation}
    \label{estim difference of minima}
    \left|  A_{1,\alpha}(g_{1,\alpha})-A_{2,\alpha}(g_{2,\alpha}) \right| \leq C \norm{\Lambda_1-\Lambda_2}{L^2 \to L^2}{}\left( \norm{f}{L^2}{2}+ \norm{g_{1,\alpha}}{L^2}{2} +\norm{g_{2,\alpha}}{L^2}{2}\right).
\end{equation}

We recall the equality
\begin{align}
    A_j(g)=\norm{u_j^g(t)-\mathds{1}_{M(\omega,s)}u_j^f(t)}{L^2}{2}+\Gamma_j.
\end{align}
We can then apply Theorem~\ref{cor approx control} which gives the existence of $g_{j}( \varepsilon) \in C^\infty_0((t-s,t)\times \omega)$ such that 
\begin{equation}
\label{ineq 2}
\norm{u_j^{g_j}(t)-\mathds{1}_{M(\omega,s)}u_j^f(t)}{L^2}{2} \leq \frac{\ve^2}{2} \norm{u^f_j(t)}{H^1}{2},
\end{equation}
with $g_j$ satisfying the bound
$$
\norm{g_{j}}{L^2}{} \leq   \ct(\ve) \norm{f}{L^2}{}.
$$
This gives
\begin{align*}
A_j(g_j) \leq \Gamma_j+ \frac{\ve^2}{2} \norm{u^f_j(t)}{H^1}{2} \Longrightarrow A_{j,\alpha}(g_j)  \leq \Gamma_j+ \frac{\ve^2}{2} \norm{u^f_j(t)}{H^1}{2}+\alpha \norm{g_j}{L^2}{2}. 
\end{align*}

We get as a consequence
$$
A_{j,\alpha}(g_{j, \alpha}) \leq A_{j,\alpha}(g_j) \leq \Gamma_j+ \frac{\ve^2}{2} \norm{u^f_j(t)}{H^1}{2}+\alpha \norm{g_j}{L^2}{2},
$$
and since $A_{j,\alpha}(g_\alpha)\geq \Gamma_j$ we have
\begin{equation}
    \label{estim minima close to gammas}
    \left| A_{j,\alpha}(g_{j, \alpha}) -\Gamma_j \right| \leq \frac{\ve^2}{2} \norm{u^f_j(t)}{H^1}{2}+\alpha \norm{g_j}{L^2}{2}.
\end{equation}
The result follows by writing
\begin{align*}
    \left| B_1(g)-B_2(g) \right| &\leq \left| \Gamma_1-\Gamma_2 \right |+ \left| B_1(g)-B_2(g)+\Gamma_1-\Gamma_2\right| \\
    &=\left| \Gamma_1-\Gamma_2 \right |+\left| A_1(g)-A_2(g) \right| \\ 
    &\leq \left | \Gamma_1-A_{1,\alpha}(g_{1,\alpha}) \right| +\left | A_{1,\alpha}(g_{1,\alpha})-A_{2,\alpha}(g_{2,\alpha}) \right| \\&\hspace{4mm}+ \left | \Gamma_2-A_{2,\alpha}(g_{2,\alpha}) \right|+\left| A_1(g)-A_2(g) \right|,
\end{align*}
using~\eqref{estim difference ofAs},\eqref{estim difference of minima} , \eqref{estim minima close to gammas} as well as the regularity of the wave equation which gives $\norm{u^f_j(t)}{H^1}{}\leq C \norm{f}{L^2}{}$.
\end{proof}

\begin{lem}
\label{lem inner products with domains of infl}
Let $T>0$, $0<s<t<T/2$, $0<\tp <T/2$ and $f,h \in L^2((0,T)\times \omega)$. Suppose that $\norm{\Lambda_1-\Lambda_2}{L^2 \to L^2}{}\leq 1$. Then for all $\varepsilon>0$ there exist $g_{j}=g_{j}(\varepsilon) \in L^2((t-s,t)\times \omega),$ $j \in \{1,2\}$ such that for all $\alpha \in (0,1)$ one has the estimate
\begin{align*}
&\left| \left( \mathds{1}_{M(\omega,s)}u_1^f(t),u_{1}^h(\tp)\right)_{L^2} -\left( \mathds{1}_{M(\omega,s)}u_2^f(t),u_{2}^h(\tp)\right)_{L^2} \right| \nonumber \\
&\leq C \norm{h}{L^2}{} \bigg(  \varepsilon \norm{f}{L^2}{}+\sqrt{\alpha}(\norm{g_1}{L^2}{}+\norm{g_2}{L^2}{})+\norm{\Lambda_1-\Lambda_2}{L^2 \to L^2}{1/2}\left(\frac{1}{\alpha}\norm{f}{L^2}{}+\norm{g_{1}}{L^2}{}\right)   \nonumber\bigg),
\end{align*}
with $C$ depending on $T,M,K,\omega,n$ and $g_j$ satisfying the bound~\eqref{bounds for control cost}.
\end{lem}

\begin{proof}
  We apply the approximate controllability Theorem~\ref{cor approx control} which gives the existence of $g_1=g_1(\varepsilon) \in C^\infty_0((t-s) \times \omega)$ satisfying 
  \begin{equation}
   \label{estim one for prop}
   \norm{u_1^{g_1}(t)-\mathds{1}_{M(\omega,s)}u^f_1(t)}{L^2}{2}\leq \ve^2 \norm{u^f_1(t)}{H^1}{2}\leq \varepsilon^2 \norm{f}{L^2}{2}.
\end{equation}
Notice that the function $g_1$ above is in fact the same as the once used in Proposition~\ref{prop difference of Bs}. 

The key point about Proposition~\ref{prop difference of Bs} is that it allows us to transfer the above estimate to $u_2$. Indeed, Proposition~\ref{prop difference of Bs} gives (with the same $g_1$)
\begin{align*}
     B_2(g_1)(t,s)&\leq B_1(g_1)(t,s)+  C \bigg(\varepsilon^2 \norm{f}{L^2}{2}+\alpha( \norm{g_{1}}{L^2}{2}+\norm{g_{2}}{L^2}{2})  \nonumber \\
    &\hspace{4mm} +\norm{\Lambda_1-\Lambda_2}{L^2 \to L^2}{}\left(\norm{g_{1,\alpha}}{L^2}{2}+\norm{g_{2,\alpha}}{L^2}{2}+\norm{g_1}{L^2}{2}+\norm{f}{L^2}{2}\right) \bigg),
\end{align*}
which using~\eqref{estim one for prop} and recalling the definition of $B_j$ in~\eqref{def of Bs} yields
\begin{align}
    \label{estim two for prop}
  \norm{u_2^{g_1}(t)-\mathds{1}_{M(\omega,s)}u^f_2(t)}{L^2}{2}&\leq \varepsilon^2 \norm{f}{L^2}{2} +  C \bigg( \varepsilon^2 \norm{f}{L^2}{2}+\alpha( \norm{g_{1}}{L^2}{2}+\norm{g_{2}}{L^2}{2})  \nonumber \\
    &\hspace{4mm} +\norm{\Lambda_1-\Lambda_2}{L^2 \to L^2}{}\left(\norm{g_{1,\alpha}}{L^2}{2}+\norm{g_{2,\alpha}}{L^2}{2}+\norm{g_1}{L^2}{2}+\norm{f}{L^2}{2}\right) \bigg).
\end{align}
We can now estimate the desired quantity as follows:
\begin{align}
&\left| \left( \mathds{1}_{M(\omega,s)}u_1^f(t),u_{1}^h(\tp)\right)_{L^2} -\left( \mathds{1}_{M(\omega,s)}u_2^f(t),u_{2}^h(\tp)\right)_{L^2} \right| \nonumber \\
&\leq \left| \left( \mathds{1}_{M(\omega,s)}u_1^f(t)-u_1^{g_1}(t),u_{1}^h(\tp)\right)_{L^2} -\left( \mathds{1}_{M(\omega,s)}u_2^f(t)-u_2^{g_1}(t),u_{2}^h(\tp)\right)_{L^2}\right| \nonumber \\
&\hspace{4mm}+\left| \left( u_1^{g_1}(t),u_{1}^h(\tp)\right)_{L^2} -\left(u_2^{g_1}(t),u_{2}^h(\tp)\right)_{L^2}\right| \nonumber \\
& \leq \norm{\mathds{1}_{M(\omega,s)}u_1^f(t)-u_1^{g_1}(t)}{L^2}{}\cdot\norm{u^h_1(\tp)}{L^2}{}+\norm{\mathds{1}_{M(\omega,s)}u_2^f(t)-u_2^{g_1}(t)}{L^2}{}\cdot \norm{u^h_2(\tp)}{L^2}{} \nonumber \\
&\hspace{4mm}+C \norm{g_1}{L^2}{}\norm{h}{L^2}{} \norm{\Lambda_1-\Lambda_2}{L^2\to L^2}{}  \nonumber \\
& \leq \norm{\mathds{1}_{M(\omega,s)}u_1^f(t)-u_1^{g_1}(t)}{L^2}{}\cdot\norm{h}{L^2}{}+\norm{\mathds{1}_{M(\omega,s)}u_2^f(t)-u_2^{g_1}(t)}{L^2}{}\cdot \norm{h}{L^2}{} \nonumber \\
& \label{term_power_one}\hspace{4mm}+C \norm{g_1}{L^2}{}\norm{h}{L^2}{} \norm{\Lambda_1-\Lambda_2}{L^2\to L^2}{},
\end{align}
where we used Proposition~\ref{prop stable recovery of inner products} and the regularity property for the wave equation. Putting together~\eqref{estim one for prop} and~\eqref{estim two for prop} yields
\begin{align*}
&\left| \left( \mathds{1}_{M(\omega,s)}u_1^f(t),u_{1}^h(\tp)\right)_{L^2} -\left( \mathds{1}_{M(\omega,s)}u_2^f(t),u_{2}^h(\tp)\right)_{L^2} \right| \nonumber \\
&\leq C \norm{h}{L^2}{} \bigg(  \varepsilon \norm{f}{L^2}{}+\sqrt{\alpha}(\norm{g_1}{L^2}{}+\norm{g_2}{L^2}{}) \nonumber \\
&\hspace{4mm}+\norm{\Lambda_1-\Lambda_2}{L^2 \to L^2}{1/2}\left( \norm{g_{1,\alpha}}{L^2}{}+\norm{g_{2,\alpha}}{L^2}{}+\norm{g_{1}}{L^2}{}+\norm{f}{L^2}{}\right)   \nonumber\bigg),
\end{align*}
where we used that since $\norm{\Lambda_1-\Lambda_2}{L^2 \to L^2}{}{}\leq 1$ the term $\norm{\Lambda_1-\Lambda_2}{L^2\to L^2}{} \leq \norm{\Lambda_1-\Lambda_2}{L^2\to L^2}{1/2} $ in~\eqref{term_power_one} can be neglected up to changing the constant $C$. The result then follows by combining the above with estimate~\eqref{estim for galpha} for $\norm{g_{j,\alpha}}{L^2}{}$.
\end{proof}

We can now reformulate the result of Lemma~\ref{lem inner products with domains of infl} in a more concise way by appropriately choosing the free parameters. This gives the following stability result.

\begin{prop}
\label{prop inner with domains of infl}
  Let $T>0$, $0<s<t<T/2$, $0<\tp <T/2$ and $f,h \in L^2((0,T)\times \omega)$. There are $\ve_0>0$ and $\ell>0$ such that for all $0<\ve \leq \ve_0$ one has the following stability estimate: 
  
  If $$\norm{\Lambda_1-\Lambda_2}{L^2\to L^2}{} \leq \left(\frac{\ve}{\ct(\ve)} \right)^\ell, 
 $$  
   then 
    $$
    \left| \left( \mathds{1}_{M(\omega,s)}u_1^f(t),u_{1}^h(\tp)\right)_{L^2} -\left( \mathds{1}_{M(\omega,s)}u_2^f(t),u_{2}^h(\tp)\right)_{L^2} \right| \leq C  \norm{f}{L^2}{} \norm{h}{L^2}{} \ve ,
    $$
    where the constant $C>0$ depends on $T,M,K, \omega, n$.
\end{prop} 

 \begin{proof}
The proof follows by estimating the right hand side of the estimate of Lemma~\ref{lem inner products with domains of infl}. Given $\ve>0$ small we choose $\sqrt{\alpha}=\frac{\ve}{\ct(\ve)}.$ The bound~\eqref{bounds for control cost} gives then $\sqrt{\alpha}(\norm{g_1}{L^2}{}+\norm{g_2}{L^2}{})\leq C \ve \norm{f}{L^2}{}$ and therefore using Lemma~\ref{lem inner products with domains of infl} we get
\begin{align*}
&\left| \left( \mathds{1}_{M(\omega,s)}u_1^f(t),u_{1}^h(\tp)\right)_{L^2} -\left( \mathds{1}_{M(\omega,s)}u_2^f(t),u_{2}^h(\tp)\right)_{L^2} \right|  \\&\leq C \norm{h}{L^2}{} \bigg(  \ve \norm{f}{L^2}{}+\norm{\Lambda_1-\Lambda_2}{L^2 \to L^2}{1/2}\left(\frac{1}{\alpha}\norm{f}{L^2}{}+\norm{g_{1}}{L^2}{}\right)   \nonumber\bigg)   \\
&\leq C \norm{h}{L^2}{} \bigg(   \ve \norm{f}{L^2}{} +\norm{\Lambda_1-\Lambda_2}{L^2 \to L^2}{1/2}\norm{f}{L^2}{} \left(\frac{1}{\alpha}+\frac{1}{\sqrt{\alpha}}\right)   \nonumber\bigg).
\end{align*}
Now if $\norm{\Lambda_1-\Lambda_2}{L^2 \to L^2}{}\leq \varepsilon^2 \alpha^2=\ve^6 \ct(\ve)^{-4} $ the above gives the desired estimate.
\end{proof}


One can iterate the results of Lemma~\ref{lem inner products with domains of infl} and Proposition~\ref{prop inner with domains of infl} to obtain inner products on intersections of domains of influence.
\begin{cor}
\label{cor with intersections}
   Let $T>0$, $0<s, s^\prime <t<T/2$, $0<\tp <T/2$ and $f,h \in L^2((0,T)\times \omega)$. Consider as well $\omega^\prime \subset \omega $ an open set. Then there are $\ve_0>0$ and $\ell>0$ such that for all $0<\ve \leq \ve_0$ one has the following stability estimate: 
  
  If $$\norm{\Lambda_1-\Lambda_2}{L^2\to L^2}{} \leq \left(\frac{\ve}{\ct(\ve)} \right)^\ell, 
 $$
   then 
    \begin{multline*}
     \left| \left( \mathds{1}_{M(\omega,s)}\mathds{1}_{M(\omega^\prime,s^\prime)}u_1^f(t),u_{1}^h(\tp)\right)_{L^2} -\left( \mathds{1}_{M(\omega,s)}\mathds{1}_{M(\omega^\prime,s^\prime)}u_2^f(t),u_{2}^h(\tp)\right)_{L^2} \right|
      \leq C  \norm{f}{L^2}{} \norm{h}{L^2}{} \ve ,   
    \end{multline*}
where the constant $C>0$ depends on $T,M,K, \omega, n$.

\end{cor}
\begin{proof}
We consider $g^\prime_1=g^\prime_1(\varepsilon) \in C^\infty_0((t-s^\prime) \times \omega^\prime)$ satisfying 
  \begin{equation*}
   \norm{u_1^{g^\prime_1}(t)-\mathds{1}_{M(\omega^\prime,s^\prime)}u^f_1(t)}{L^2}{2}\leq \varepsilon^2 \norm{f}{L^2}{2},
\end{equation*} 
and write
\begin{align*}
    &\left| \left( \mathds{1}_{M(\omega,s)}\mathds{1}_{M(\omega^\prime,s^\prime)}u_1^f(t),u_{1}^h(\tp)\right)_{L^2} -\left( \mathds{1}_{M(\omega,s)}\mathds{1}_{M(\omega^\prime,s^\prime)}u_2^f(t),u_{2}^h(\tp)\right)_{L^2} \right| \\
    & \leq \bigg| \left( \mathds{1}_{M(\omega,s)} \mathds{1}_{M(\omega^\prime,s^\prime)}u_1^f(t)- \mathds{1}_{M(\omega,s)}u_1^{g^\prime_1}(t),u_{1}^h(\tp)\right)_{L^2}\\
    &\hspace{4mm}- \left( \mathds{1}_{M(\omega,s)} \mathds{1}_{M(\omega^\prime,s^\prime)}u_2^f(t)- \mathds{1}_{M(\omega,s)}u_2^{g_1}(t),u_{2}^h(\tp)\right)_{L^2}\bigg| \\
&\hspace{4mm}+\left| \left( \mathds{1}_{M(\omega,s)}u_1^{g_1}(t),u_{1}^h(\tp)\right)_{L^2} -\left( \mathds{1}_{M(\omega,s)}u_2^{g_1}(t),u_{2}^h(\tp)\right)_{L^2}\right|.
\end{align*}
We can then apply Lemma~\ref{lem inner products with domains of infl} directly for the second term and follow the same steps as in the proof of Lemma~\ref{lem inner products with domains of infl} for the first one. The corollary follows.
\end{proof}

Next we recover the point values.

\begin{prop}
 \label{prop point values}
  Let $T>2 \mathcal{L}(K, \omega)$ and $f,h \in H_0^k((0,T)\times \omega)$ with $k > n/2$. There are constants $\ve_0>0$ and $\ell=\ell(n)$ such that for all $0<\ve \leq \ve_0$ one has the following stability estimate: 
  
  If $$\norm{\Lambda_1-\Lambda_2}{L^2\to L^2}{} \leq \left(\frac{\ve}{\ct(\ve^\ell)} \right)^\ell, 
 $$
   then 
    $$
    \left| u_1^f(t,x) \overline{u_{1}^h(\tp,x)} -u_2^f(t,x) \overline{u_{2}^h(\tp,x)} \right| \leq  C\norm{f}{H^k}{} \norm{h}{H^k}{}\ve ,
    $$
   for all $0<t, \tp <T/2$ and $x \in K$ with a constant $C>0$ that depends only on $T,M,K, \omega, n$.
\end{prop}

\begin{proof}
Let $x \in K$. We denote $\dist(x ,\omega)=s$ and consider $y \in \partial \omega$ such that $\dist(x,y)=s$. Notice that $s$ moves on the compact interval $s \in [s_0,s_1]$ with $s_0= \dist(K, \omega)>0$ and $s_1= \mathcal{L}(K, \omega) <T/2$. We now want to apply Corollary~\ref{cor with intersections} to domains that eventually collapse to a point in order to transfer the estimate from inner products to point values. To do so we consider the set $A(\eta)$ depending on a small parameter $\eta>0$ defined by
$$
A(\eta):= M(B(-r \nu(y), r), s+\eta) \backslash M(B(-2r \nu(y),2r), s- \eta),
$$
where we denote by $\nu(y)$ the outward unit normal vector at $y \in \partial \omega$ and by $B(z,r)$ the ball of center $z$ and radius $r$. The radius $r$ is chosen sufficiently small that $B(-r \nu(y), r) \subset \omega$ and it can be chosen uniformly for any $y \in  \partial \omega$ by taking $r \leq \frac{r_0}{4}$ with $r_0$ the injectivity radius of $\omega$ in boundary normal coordinates. Remark that in fact
$$
A(\eta)=B(-r \nu(y), s+r+\eta)\backslash B(-2r \nu(y),2r+s-\eta).
$$

To alleviate notation we define 
$$
\ff(x):=u_1^f(t,x) \overline{u_{1}^h(\tp,x)} -u_2^f(t,x) \overline{u_{2}^h(\tp,x)},
$$
and estimate
\begin{align}
\label{two terms for estim of ff}
|\ff(x)|\leq \left|\frac{1}{|A(\eta)|} \int_{A(\eta)} \ff(z)-\ff(x)dz \right|+\left|\frac{1}{|A(\eta)|} \int_{A(\eta)}\ff(z)dz \right |,
\end{align}
where we denote by $|A(\eta)|$ the volume of the set $A(\eta)$. 

We start with the second term in~\eqref{two terms for estim of ff}. We remark that since $\omega$ is smooth then for small $\eta$ the set $A(\eta)$ behaves like a hyperspherical cap with a height that scales in $\eta$ and a radius that scales in $\sqrt{\eta}$ (see Figure~\ref{hyper cap}). Consequently there is a constant $C$ (depending on $s \in [s_0,s_1], \omega$ and $n$) such that $|A(\eta)| \sim C \eta^{\frac{n+1}{2}}$. Given $\ve>0$ small we can then apply Corollary~\ref{cor with intersections} to find that $\norm{\Lambda_1-\Lambda_2}{L^2\to L^2}{} \leq \left(\frac{\ve}{\ct(\ve)} \right)^\ell$ implies $\left| \int_{A(\eta)}\ff(z)dz \right | \leq C \norm{h}{L^2}{}   \norm{f}{L^2}{} \ve$ and therefore
\begin{equation}
\label{second term for ff}
\left|\frac{1}{|A(\eta)|} \int_{A(\eta)}\ff(z)dz \right | \leq C\frac{ \norm{h}{L^2}{}   \norm{f}{L^2}{} \ve}{\eta^{\frac{n+1}{2}}}.
\end{equation}
Notice that we were ably to apply Corollary~\ref{cor with intersections} thanks to the assumption $T>2\mathcal{L}(K, \omega)$ which guarantees that $s < T/2$.

To control the first term we notice that the solutions $u_j(t)$ lie in $H^{k+1}$ and consequently for $k>n/2$ we have by the Sobolev embedding that they are $C^1$ and in particular $C_{t,f,h}$-Lipschitz. The Lipschitz constant $C_{t,f,h}$ can be bounded from above by $C_{t,f,h} \leq \sum_j \norm{\nabla (u_j^f(t,x)u_j^h(t^\prime,x)) }{L^\infty}{} \leq C \norm{f}{H^k}{} \norm{h}{H^k}{}$ for a constant $C$ depending only on $M,K,\omega,T$. This is thanks to higher regularity estimates for the wave equation(see for instance~\cite[Theorem 6, Chapter 7.2]{Evans:98}). This gives
\begin{align}
   &\left|\frac{1}{|A(\eta)|} \int_{A(\eta)} \ff(z)-\ff(x)dz \right| \leq  \frac{C \norm{f}{H^k}{} \norm{h}{H^k}{}}{|A(\eta)|} \int_{A(\eta)}|x-z| dz  \nonumber \\  & \leq \frac{ C\norm{f}{H^k}{} \norm{h}{H^k}{} }{|A(\eta)|} \int_{A(\eta)}\sqrt{\eta} dz 
   \label{first term for ff}
   \leq C \norm{f}{H^k}{} \norm{h}{H^k}{} \sqrt{\eta},
\end{align}
where we used the fact that for small $\eta$ the set $A(\eta)$ is contained in a ball centered at $x$ with a radius $C \sqrt{\eta}$ where the constant $C$ depends only on $n,T,K,\omega$. Putting together~\eqref{second term for ff} and~\eqref{first term for ff} and choosing $\ve=\eta^{\frac{n+2}{2}}$ yields then
$$
|\ff(x)|\leq C\norm{f}{H^k}{} \norm{h}{H^k}{}  \ve^{\frac{1}{n+2}} +C\norm{f}{H^k}{} \norm{h}{H^k}{} \ve^{\frac{1}{n+2}},
$$
and the proposition follows by changing $\ve$ to $\ve^{n+2}.$
\end{proof}

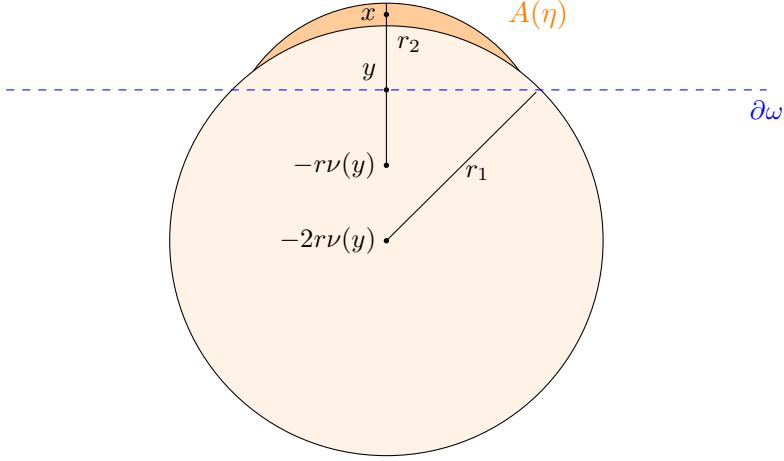
\begin{figure}
    \centering
    \begin{tikzpicture}
  
   \coordinate (O) at (0,0) ;

    \filldraw[fill=orange!40] (0,1) circle (2+0.15);
   \filldraw[fill=orange!10, semitransparent] (0,0) circle (3-0.15);


  \draw(0,0)--(1.97,1.97) ; 
  \draw(0,1)--(0,3+0.15)  ;
 \draw node at (0.3,2.6)[font=\small] {$r_2$} ;
  \draw node at (1.2,0.9)[font=\small] {$r_1$} ;

  \draw node at (2, 3) [orange!100]  {$A(\eta)$};

   \fill (0, 0)  circle[radius=1pt]   node [left, font= \small]  {$-2r \nu(y)$};
     \fill (0, 1)  circle[radius=1pt]   node [left, font= \small]  {$-r \nu(y)$};
       \fill (0, 2)  circle[radius=1pt]   node [above left, font= \small]  {$y$};
         \fill (0, 3)  circle[radius=1pt]   node [left, font= \small]  {$x$};
    
\draw[blue, dashed](-5,2)--(5,2)  node[ pos=1,below, font=\small] {$\partial \omega$} ;

 \end{tikzpicture}

    \caption{The domain $A(\eta)$ in darker color, used in the proof of Proposition~\ref{prop point values}. As $\eta \to 0$  $A(\eta)$ collapses to the point $x$. The two balls correspond to the two domains of influence under consideration and the radii are $r_1=2r+s-\eta$ and $r_2=r+s+\eta$.}
    \label{hyper cap}
\end{figure}

\subsection{Geometric optics solutions}

We are now ready to recover the difference of two solutions $u^f_1-u^f_2$ associated to the potentials $q_1,q_2$. To do so, we need to apply Proposition~\ref{prop point values} to some well chosen sources $h$ that correspond to localised solutions. These are the geometric optics solutions and their construction is classical. Here we essentially follow the presentation in~\cite{kian2019recovery,Kian2019, nursultanov2023introduction}.

\begin{prop}
    \label{prop diff of sol}
    Let $T>2 \mathcal{L}(K, \omega)$ and $f\in H_0^k((0,T)\times \omega)$ with $k > n/2$. There are  constants $\ve_0>0$ and $\ell=\ell(n)>0$ such that for all $0<\ve \leq \ve_0$ one has the following stability estimate: 
  
  If $$\norm{\Lambda_1-\Lambda_2}{L^2\to L^2}{} \leq \left(\frac{\ve}{\ct(\ve^\ell)} \right)^\ell,
 $$  
   then 
    $$
    \left| u_1^f(t,x)  -u_2^f(t,x)  \right| \leq  C\norm{f}{H^k}{}\ve ,
    $$
   for all $0<t<T/2$ and $x \in K$ with a constant $C>0$ that depends only on $T,M,K, \omega, n$.
\end{prop}

\begin{proof}
Let $x_0 \in K$. We consider as before $y \in \partial \omega$ with $\dist(x ,\omega)=\dist(x_0,y):=s$. We consider the outward unit normal vector $\nu(y)$ at $y \in \partial \omega$ and define the line $\beta(t)=x_0-(s+\delta-t)\nu(y)$ for $\delta>0$ small.  We have in particular $\beta(0) \in \omega$, $\beta(s+\delta)=x_0$ and $\beta(\delta)=y$.

Consider now a cut-off $\chi \in C^{\infty}_0((B(0,2 \eta); [0,1])$ with $\chi=1$ on $B((0, \eta))$. For $\eta=\eta(\delta)>0$ sufficiently small the function $a(t,x):=\chi(x - \beta(t))$ satisfies $\supp(a(0, \cdot)) \subset \omega$. We consider as well another cut-off $\tchi \in C^\infty_0(\omega)$ such that $\tchi=1$ on the support of $a(0,\cdot)$. We define then 
$$
v_j(t,x)=a(t,x)e^{i \sigma (t-x \cdot \nu(y))}+R_{j, \sigma}(t,x),
$$
where $\rj$ satisfies 
\begin{equation*}
    \begin{cases}
    (\p^2_t-\Delta+q_j)\rj=-e^{i \sigma (t-x \cdot \nu(y))} (\partial^2_t-\Delta+q_j(1 - \tchi(x)))a & \textnormal{in} \:(0,T) \times \R^n,\\
    \rj(0, \cdot)=\p_t \rj(0, \cdot)=0  & \textnormal{in} \: \R^n.
    \end{cases}
\end{equation*}
Define finally $\tv_j:=v_j- \tchi(x) a(t,x)e^{i \sigma (t-x \cdot \nu(y))}=(1- \tchi)a(t,x)e^{i \sigma (t-x \cdot \nu(y))}+R_{j, \sigma}(t,x)$. As in~\cite[Lemma 3.5]{Kian2019} we have that $(\p^2_t-\Delta+q_j) v_j=0$. Moreover, the support properties of $a$ and $\tchi$ combined with $ \rj(0, \cdot)=\p_t \rj(0, \cdot)=0 $ imply that $\tv_j(0, \cdot)= \p \tv_j(0, \cdot)=0$. We find then that $\tv_j$ satisfy
\begin{equation*}
    \begin{cases}
    (\p^2_t-\Delta+q_j)\tv_j=-(\partial^2_t-\Delta) \left(\tchi(x)a(t,x) e^{i \sigma (t-x \cdot \nu(y))}  \right) & \textnormal{in} \:(0,T) \times \R^n,\\
    \tv_j(0, \cdot)=\p_t \tv_j(0, \cdot)=0  & \textnormal{in} \: \R^n.
    \end{cases}
\end{equation*}
This means that $\tv_j=u^h_j$ with $h:=-(\partial^2_t-\Delta) \left(\tchi(x)a(t,x) e^{i \sigma (t-x \cdot \nu(y))}  \right)$. In particular we have that $\norm{h}{H^k}{}\leq C \frac{\sigma^{k+2}}{\delta^{k+2}}$. Applying Proposition~\ref{prop point values} yields that under the assumption $\norm{\Lambda_1-\Lambda_2}{L^2\to L^2}{} \leq \left(\frac{\ve}{\ct(\ve^\ell)} \right)^\ell$ one has
\begin{equation}
 \label{estim after appl prop}
  \left| u_1^f(t,x) \overline{\tv_1(\tp,x)} -u_2^f(t,x) \overline{\tv_2(\tp,x)} \right| \leq  C\norm{f}{H^k}{} \frac{\sigma^{k+2} \ve}{\delta^{k+2}}.
\end{equation}
Let $\psi \in C^\infty_0((s+\delta-2\eta, s+\delta+2\eta);[0,1])$ with $\psi=1$ on $[s+\delta-\eta, s+\delta+\eta]$. Then multiplying~\eqref{estim after appl prop} by $\psi(\tp)$ and integrating in space and time gives
\begin{align}
\label{estim with psi}
    \int_0^T \int_{K }  \left| u_1^f(t,x) \overline{\tv_1(\tp,x)} -u_2^f(t,x) \overline{\tv_2(\tp,x)} \right| \psi(\tp) d \tp d x \leq C\norm{f}{H^k}{} \frac{\sigma^{k+2} \ve}{\delta^{k+2}},
\end{align}
where $C$ depends on $T,K$.
We estimate then as follows:
\begin{align}
 &\int_0^T \int_{K \backslash \omega} \left| (u^f_1(t,x)-u_2^f(t,x))a(\tp,x)  \right| \psi(\tp) d \tp dx  \nonumber \\
 &\leq      \int_0^T \int_{K } \left| u_1^f(t,x) \overline{\tv_1(\tp,x)} -u_2^f(t,x) \overline{\tv_2(\tp,x)} \right| \psi(\tp) d \tp dx \nonumber \\& 
 \label{two terms for integral}
 \hspace{4mm}+
\int_0^T \int_{K } \left| u_1^f(t,x)\overline{R_{1, \sigma}}(\tp,x)- u_2^f(t,x)\overline{R_{2, \sigma}}(\tp,x)    \right| \psi(\tp) d \tp dx.
\end{align}
 For the second term we notice that as in~\cite[Lemma 2.2]{kian2019recovery} we have that $\norm{\rj}{L^2}{} \to 0$ as $\sigma \to + \infty$ and in fact \footnote{One just needs to use a quantitative version of the Riemann-Lebesgue lemma which follows after integrating by parts once. The rest of the proof remains the same.} we have
 $$
 \norm{\rj}{L^2}{} \leq \frac{C}{\sigma}.
 $$
We obtain as a consequence:
\begin{align}
 &\int_0^T \int_{K } \left| u_1^f(t,x)\overline{R_{1, \sigma}}(\tp,x)- u_2^f(t,x)\overline{R_{2, \sigma}}(\tp,x)    \right| \psi(\tp) d \tp dx  \nonumber \\ \label{estim with rj} &\leq C \norm{u_1^f(t, \cdot)}{L^2}{} \cdot \norm{R_{1, \sigma}}{L^2}{}+C \norm{u_2^f(t, \cdot)}{L^2}{} \cdot \norm{R_{2, \sigma}}{L^2}{}\leq  \frac{C \norm{f}{L^2}{}}{\sigma}.
\end{align}
Putting~\eqref{estim with psi} and~\eqref{estim with rj} in~\eqref{two terms for integral} we get
\begin{align}
\label{final estim with integral}
\int_0^T \int_{K \backslash \omega} \left| (u^f_1(t,x)-u_2^f(t,x))a(\tp,x)  \right| \psi(\tp) d \tp dx \leq  C \norm{f}{H^k}{}\left(\frac{\sigma^{k+2} \ve}{\delta^{k+2}} +\frac{1}{\sigma}\right).    
\end{align}
We recall that the support of $a(\tp,x) \psi(\tp)$ localizes arbitrarily close to $\{x_0\}$ as $\eta \to 0$. We can use the same argument as in the proof of Proposition~\ref{prop point values} to pass from the integral to a point value estimate. To do so we use the regularity of $u^f_j$ and lose some powers of $\eta$ depending on the dimension only. Choosing the parameters $\delta, \sigma, \eta$ as functions of $\ve$ we find then that~\eqref{final estim with integral} implies
 $$
    \left| u_1^f(t,x_0)  -u_2^f(t,x_0)  \right| \leq  C\norm{f}{H^k}{}\ve^{c^{-1}(n)},
    $$
for a constant $C>0$ depending only on $T,K,M,n$ and $c(n)>1$ depending on the dimension only. Since $x_0$ is an arbitrary point in $K$ the proposition follows.
\end{proof}

\begin{figure}
    \centering
    \begin{tikzpicture}

 \filldraw[blue!10] plot [smooth cycle, tension=1] coordinates {(-2,0) (-1,2) (2,1) };
 \draw plot [smooth cycle, tension=1] coordinates {(-2,0) (-1,2) (2,1) };

\draw node at (2,2)[blue, font=\small] {$\omega$} ;

\fill (-0.5, 4.5)  circle[radius=1pt]   node [left, font= \small]  {$x_0$};

\fill (-0.5,2.07)  circle[radius=1pt]  ;
 \draw node at (-0.3,2.3)[ font=\small] {$y$} ;

\fill[brown] (-0.5, 4)  circle[radius=0.15]   node [left=1.5, font= \small]  {$\supp(a(t,\cdot))$};

\fill[brown] (-0.5, 1.5)  circle[radius=0.15]   node [left=1.5, font= \small]  {$\supp(a(0,\cdot))$};

\fill (-0.5, 1.5)  circle[radius=1pt]   node [left=0.3, font= \small]  {$\beta(0)$};

\draw [red] (-0.5,1.5)--(-0.5,4.5) node [right=0.5, font= \small]  {$\beta(t)$}  ; 

  \draw [-latex](-0.5,2.07) -- (-0.5,3) node [right, font=\small]{$\nu(y)$};

 \end{tikzpicture}

    \caption{Support of the geometric optics solution, localizing close to the line $\beta(t)$.}
    \label{fig support of a }
\end{figure}
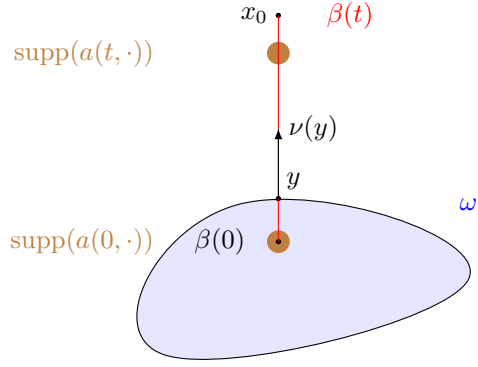

\begin{remark}
\label{rem on optics higher order}
    For the proof of Proposition~\ref{prop diff of sol} we only needed to use geometric optics solution of zeroth order. This has the advantage that it does not consume any regularity for the potential $q$ other than $q \in L^\infty$. However, to conclude the proof of Theorem~\ref{main thm} we will need some geometric optics solutions with better estimates for the error term $R_{j,\sigma}$. This is achieved by using some higher order approximation. The construction is as follows (see~\cite[Section 4.2]{nursultanov2023introduction} for more details).
    We define 
    $$
    w_j=e^{i \sigma (t-x \cdot \nu(y))}A_N+\rp_{j,s}, \quad A_N=\sum_{l=0}^{N}a_l\sigma^{-l},
    $$
  where $a_0=a$ is the same as the one in the proof of Proposition~\ref{prop diff of sol} and $a_l$ for $l\geq 1$ are defined inductively as solutions of the transport equations
  $$
  \p_t a_l+\nu(y) \cdot \nabla_x a_l-\frac{i}{2}(\p^2_t-\Delta+q_j)a_{l-1}=0.
  $$
This gives 
\begin{equation}
\label{ansatz error}
(\p^2_t-\Delta+q_j)(e^{i \sigma (t-x \cdot \nu(y))}A_N)\leq \frac{C}{\sigma^N},
\end{equation}
for a constant $C$ that depends on $\p^{|\alpha|}_x q_j$ for $|\alpha|\leq 2N$. We obtain then the solutions $\tw_j=u_j^{\thh}=(1-\tchi)e^{i \sigma (t-x \cdot \nu(y))}A_N+\rp_{j,\sigma}$ with $\thh=-(\p^2_t-\Delta)(\tchi e^{i \sigma (t-x \cdot \nu(y))}A_N),$ and an error satisfying
\begin{equation*}
    \begin{cases}
    (\p^2_t-\Delta+q_j)\rp_{j,\sigma}=-(\p^2_t-\Delta+q_j)(e^{i \sigma (t-x \cdot \nu(y))}A_N) & \textnormal{in} \:(0,T) \times \R^n,\\
    \rp_{j,\sigma}(0, \cdot)=\p_t \rp_{j,\sigma}(0, \cdot)=0  & \textnormal{in} \: \R^n.
    \end{cases}
\end{equation*}
Using higher regularity estimates for the wave equation and~\eqref{ansatz error} we find then
\begin{equation}
  \label{estim for error high order}
  \norm{\rp_{j,\sigma}}{H^k}{}\leq \norm{(\p^2_t-\Delta+q_j)(e^{i \sigma (t-x \cdot \nu(y))}A_N)}{H^{k-1}}{} \leq C \sigma^{k-1-N},
\end{equation}
with $C$ depending on $\p^{|\alpha|}_x q$ for $|\alpha|\leq \max (2N,k-1)$.

\end{remark}

\subsection{End of the proofs}
\label{section_end_proofs}

\begin{proof}[Proof of Theorem~\ref{main thm}]
    To conclude the proof of the main theorem one needs finally to transfer the estimate of Proposition~\ref{prop diff of sol} concerning the difference of the solutions $u_1,u_2$ to an estimate for the difference of the potentials $q_1,q_2$. To do so we start by writing
\begin{align}
    &(\p^2_t-\Delta+q_1)u^f_1=f=(\p^2_t-\Delta+q_2)u^f_2  \nonumber \\
    \label{equality for potentials}
    &\Longrightarrow u^f_2(q_2-q_1)=(\p_t^2 -\Delta)(u^f_1-u^f_2)+q_1(u^f_1-u^f_2).
\end{align}
Given now $x_0 \in K$ we consider the geometric optics solutions of Remark~\ref{rem on optics higher order}, $\tw_2=u_2^{\thh}$ and estimate
\begin{equation}
\label{estim positive term}
|\tw_2(\sd,x_0)|=\left|e^{i \sigma (\sd-x_0 \cdot \nu(y))}A_N+\rp_{2, \sigma}(\sd,x_0) \right| \geq 1-\frac{C}{\sigma}-|\rp_{2, \sigma}(\sd,x_0)|,
\end{equation}
where we write $\sd=s+\delta$. Using that $\norm{q_j}{C^m}{}\leq M$ for $m\geq 2N$ we see that the constant $C$ above depends on $M,K, \omega,T,n$ only. We remark as well that $\delta>0$ can be chosen uniformly for all $x \in K$.\footnote{The condition on $\delta$ is that it has to be smaller than the injectivity radius of $\omega$ in boundary normal coordinates.} We have for $k>n/2$:
$$
\norm{\rp_{2, \sigma}(\sd)}{C^0}{}\leq \norm{\rp_{2, \sigma}(\sd)}{H^k}{} \leq C \sigma^{k-1-N},
$$
where we used~\eqref{estim for error high order}. Taking $N\geq k$ implies that $\norm{\rp_{2, \sigma}(\sd)}{C^0}{} \to 0$ as $\sigma \to \infty$. This combined with~\eqref{estim positive term} gives the existence of $\sigma_0$ such that for $\sigma\geq \sigma_0$ one has
\begin{equation}
 \label{condition for sigma}
 |\tw_2(\sd,x_0)| \geq \frac{1}{2}.
\end{equation}
Using this in~\eqref{equality for potentials} for $f=\thh$ yields
\begin{align}
\label{point wise ineq for pot}
    |q_2(x_0)-q_1(x_0)| &\leq CM |u_1^{\thh}(\sd,x_0)-u_2^{\thh}(\sd,x_0)| \nonumber \\ &\hspace{4mm}+C|(\p_t^2 -\Delta)(u_1^{\thh}(\sd,x_0)-u^{\thh}_2(\sd,x_0))|.
\end{align}
Since the constant $C$ in~\eqref{point wise ineq for pot} is independent of $x_0 \in K$ the inequality holds for all $x \in K$. We integrate then over $K$ and bound as follows:
\begin{align}
    \norm{q_2-q_1}{L^2(K)}{} &\leq CM \norm{u_1^{\thh}-u^h_2}{L^2((0,T/2)\times K)}{}+C \norm{(\p_t^2 -\Delta)(u_1^{\thh}-u^{\thh}_2)}{L^2((0,T/2)\times K)}{} \nonumber\\
    &\leq C \norm{u_1^{\thh}-u^{\thh}_2}{L^2((0,T/2)\times K)}{}+C\norm{u_1^{\thh}-u^{\thh}_2}{L^2((0,T/2)\times K)}{\frac{1}{3}} \cdot \norm{u_1^{\thh}-u^{\thh}_2}{H^3((0,T/2)\times K)}{\frac{2}{3}} \nonumber \\
    &  \label{first estim for pot} \leq C \norm{u_1^{\thh}-u_2^{\thh}}{L^2((0,T/2)\times K)}{\frac{1}{3}} \norm{\thh}{H^2}{\frac{2}{3}},
    \end{align}
with $C$ depending on $M,T,K$ and thanks to the assumption $\norm{q_j}{C^m(K)}{}\leq M$. We used as well the interpolation inequality $\norm{\cdot}{H^2}{} \leq C \norm{\cdot}{L^2}{1/3} \norm{\cdot}{H^3}{2/3}$ and higher regularity estimates for the wave equation. Given now $\ve>0$ we apply Proposition~\ref{prop diff of sol} which yields that under the assumption $\norm{\Lambda_1-\Lambda_2}{L^2\to L^2}{} \leq \left(\frac{\ve}{\ct(\ve^\ell)} \right)^\ell$ one has
\begin{equation}
\label{second estim for pot}
\norm{u^{\thh}_1-u^{\thh}_2}{L^2((0,T/2)\times K)}{} \leq C \norm{\thh}{H^k}{} \ve.
\end{equation}
 We recall that $\norm{\thh}{H^k}{}$ can be bounded by above by a constant depending only on $M,K, \omega, T, n$. Putting together~\eqref{first estim for pot} with~\eqref{second estim for pot} implies that if $\norm{\Lambda_1-\Lambda_2}{L^2\to L^2}{} \leq \left(\frac{\ve}{\ct(\ve^\ell)} \right)^\ell$ then
\begin{equation*}
    \norm{q_2-q_1}{L^2(K)}{} \leq C \ve^{1/3}, 
\end{equation*}
for a constant $C$ depending only on $T,M,K,\omega$ and $n$. This proves the stability estimate of Theorem~\ref{main thm}, up to changing the power of $\ve$. 

Finally, the condition on the potentials is $\norm{q_j}{C^m}{}\leq M$ with $m\geq 2N$, $2N\geq 2k$ and $k> \frac{n}{2}$. This gives the condition $m > n$ on $m$ and the proof of Theorem~\ref{main thm} is complete. 
\end{proof}

We finally explain how Theorem~\ref{thm_inverse_half} is obtained.

\begin{proof}[Proof of Theorem~\ref{thm_inverse_half}]

The key point is that for the proof of Theorem~\ref{main thm} one in fact only needs to use Theorem~\ref{cor approx control} in the case where $\omega$ is a ball (this comes from the proof of Proposition~\ref{prop point values}). In the situation where observation takes place on the half-space $H$ one needs to use balls centered at some point of $H$ such that their domains of influence meet all the points of $K$. Consider then $\rtt$ with $ \frac{1}{2} \dist(K, H) \leq  \frac{\rtt}{2} \leq 2( \dist (K,H)+\textnormal{diam}(K))$. In particular $\rtt \in [\frac{1}{r_0},r_0]$ for some $r_0$ depending only on $K$ and we impose $T>2r_0$. We can now use as domains $M(\omega,s)$ some balls with radius $3/2 \rtt$ for $\rtt \in [\frac{1}{r_0},r_0]$. In this situation, Theorem 1.2 in~\cite{FO25explicit} (combined with the cost of control result of~\cite{Robbiano:95}) gives the explicit constant $\ct(\ve)=e^{e^{C/\ve}}$ in Theorem~\ref{main thm}. We obtain that if
$$
\norm{\Lambda_1-\Lambda_2}{L^2\to L^2}{} \leq \exp({-e^{C/\ve^\ell}}),
$$
then 
$$\norm{q_2-q_1}{L^2(K)}{}\leq C \ve,$$
where $C>0,\ell>1$ depend only on $T,K,M,n$. The statement then follows by choosing $\ve$ in terms of $\norm{\Lambda_1-\Lambda_2}{L^2\to L^2}{}$.
\end{proof}

   \bigskip

\noindent \textbf{Acknowledgements} The authors were supported by the European Research Council 
of the European Union, grant 101086697 (LoCal), 
and the Research Council of Finland, grants 347715,
353096 (Centre of Excellence of Inverse Modelling and Imaging) 
and 359182 (Flagship of Advanced Mathematics for Sensing Imaging and Modelling). 
Views and opinions expressed are those of the authors only and do not 
necessarily reflect those of the European Union or the other funding 
organizations.

\small \bibliographystyle{alpha}

\bibliography{bibl}

\begin{thebibliography}{BKLS08}

\bibitem[AG22]{arnaiz2022stability}
V{\'\i}ctor Arnaiz and Colin Guillarmou.
\newblock Stability estimates in inverse problems for the {S}chr{\"o}dinger and
  wave equations with trapping.
\newblock {\em Revista Matem{\'a}tica Iberoamericana}, 39(2):495--538, 2022.

\bibitem[ASS90]{alessandrini1990stability}
Giovanni Alessandrini, John Sylvester, and Z~Sun.
\newblock Stability for a multidimensional inverse spectral theorem.
\newblock {\em Communications in Partial Differential Equations},
  15(5):711--736, 1990.

\bibitem[Bel87]{Belishev:87}
Mikhail Belishev.
\newblock An approach to multidimensional inverse problems for the wave
  equation.
\newblock {\em Dokl. Akad. Nauk SSSR}, 297(3):524--527, 1987.

\bibitem[BF11]{MouradDos}
Mourad Bellassoued and David Dos~Santos Ferreira.
\newblock Stability estimates for the anisotropic wave equation from the
  {D}irichlet-to-{N}eumann map.
\newblock {\em Inverse Problems and Imaging}, 5(4):745--773, 2011.

\bibitem[BG97]{BG:97}
Nicolas Burq and Patrick G{\'e}rard.
\newblock Condition n\'ecessaire et suffisante pour la contr\^olabilit\'e
  exacte des ondes.
\newblock {\em C. R. Acad. Sci. Paris S\'er. I Math.}, 325(7):749--752, 1997.

\bibitem[BILL24]{burago2024quantitativestabilitygelfandsinverse}
Dmitri Burago, Sergei Ivanov, Matti Lassas, and Jinpeng Lu.
\newblock Quantitative stability of {G}el'fand's inverse boundary problem.
\newblock {\em arXiv:2012.04435}, 2024.

\bibitem[BKL16]{BKL:16}
Roberta Bosi, Yaroslav Kurylev, and Matti Lassas.
\newblock Stability of the unique continuation for the wave operator via
  {T}ataru inequality and applications.
\newblock {\em J. Differential Equations}, 260(8):6451--6492, 2016.

\bibitem[BKL17]{bosi2017reconstruction}
Roberta Bosi, Yaroslav Kurylev, and Matti Lassas.
\newblock Reconstruction and stability in {G}el'fand's inverse interior
  spectral problem.
\newblock {\em arXiv preprint arXiv:1702.07937}, 2017.

\bibitem[BKL18]{bosi2018stability}
Roberta Bosi, Yaroslav Kurylev, and Matti Lassas.
\newblock Stability of the unique continuation for the wave operator via
  {T}ataru inequality: the local case.
\newblock {\em Journal d'Analyse Math{\'e}matique}, 134(1):157--199, 2018.

\bibitem[BKLS08]{Bingham2008IterativeTC}
Kenrick Bingham, Yaroslav Kurylev, Matti Lassas, and Samuli Siltanen.
\newblock Iterative time-reversal control for inverse problems.
\newblock {\em Inverse Problems and Imaging}, 2:63--81, 2008.

\bibitem[Bla67]{blagoveshchenskii1967inverse}
AS~Blagoveshchenskii.
\newblock The inverse problem in the theory of seismic wave propagation.
\newblock In {\em Spectral Theory and Wave Processes}, pages 55--67. Springer,
  1967.

\bibitem[BLR92]{BLR:92}
Claude Bardos, Gilles Lebeau, and Jeffrey Rauch.
\newblock Sharp sufficient conditions for the observation, control, and
  stabilization of waves from the boundary.
\newblock {\em SIAM J. Control Optim.}, 30:1024--1065, 1992.

\bibitem[Eva98]{Evans:98}
Lawrence~C. Evans.
\newblock {\em Partial differential equations}.
\newblock Graduate Studies in Mathematics. American Mathematical Society,
  Providence, RI, 1998.

\bibitem[FO25]{FO25explicit}
Spyridon Filippas and Lauri Oksanen.
\newblock On the blowup of quantitative unique continuation estimates for waves
  and applications to stability estimates.
\newblock {\em arXiv:2502.13040}, 2025.

\bibitem[H{\"o}r92]{Hormander:92}
Lars H{\"o}rmander.
\newblock A uniqueness theorem for second order hyperbolic differential
  equations.
\newblock {\em Comm. Partial Differential Equations}, 17(5-6):699--714, 1992.

\bibitem[H{\"o}r97]{Hor:97}
Lars H{\"o}rmander.
\newblock On the uniqueness of the {C}auchy problem under partial analyticity
  assumptions.
\newblock In {\em Geometrical optics and related topics ({C}ortona, 1996)},
  volume~32 of {\em Progr. Nonlinear Differential Equations Appl.}, pages
  179--219. Birkh\"auser Boston, Boston, MA, 1997.

\bibitem[KKL01]{bookMatti}
Alexander Kachalov, Yaroslav Kurylev, and Matti Lassas.
\newblock {\em Inverse boundary spectral problems}.
\newblock Chapman and Hall/CRC, 2001.

\bibitem[KMO19]{Kian2019}
Yavar Kian, Morgan Morancey, and Lauri Oksanen.
\newblock Application of the boundary control method to partial data
  {B}org-{L}evinson inverse spectral problem.
\newblock {\em Mathematical Control and Related Fields}, 9(2):289--312, 2019.

\bibitem[KO19]{kian2019recovery}
Yavar Kian and Lauri Oksanen.
\newblock Recovery of time-dependent coefficient on {R}iemannian manifold for
  hyperbolic equations.
\newblock {\em International Mathematics Research Notices},
  2019(16):5087--5126, 2019.

\bibitem[LL19]{Laurent_2018}
Camille Laurent and Matthieu L\'{e}autaud.
\newblock Quantitative unique continuation for operators with partially
  analytic coefficients. {A}pplication to approximate control for waves.
\newblock {\em J. Eur. Math. Soc. (JEMS)}, 21(4):957--1069, 2019.

\bibitem[LL23]{LL:23notes}
Camille Laurent and Matthieu L\'eautaud.
\newblock Lectures on unique continuation for waves.
\newblock {\em submitted, arxiv.org/pdf/2307.02155}, 2023.

\bibitem[NO23]{nursultanov2023introduction}
Medet Nursultanov and Lauri Oksanen.
\newblock Introduction to inverse problems for {H}yperbolic {PDE}s.
\newblock {\em arxiv 2311.06020}, 2023.

\bibitem[PSU23]{paternain2023geometric}
Gabriel~P Paternain, Mikko Salo, and Gunther Uhlmann.
\newblock {\em Geometric inverse problems}, volume 204.
\newblock Cambridge University Press, 2023.

\bibitem[Rob91]{Robbiano:91}
Luc Robbiano.
\newblock Th\'eor\`eme d'unicit\'e adapt\'e au contr\^ole des solutions des
  probl\`emes hyperboliques.
\newblock {\em Comm. Partial Differential Equations}, 16(4-5):789--800, 1991.

\bibitem[Rob95]{Robbiano:95}
Luc Robbiano.
\newblock Fonction de co\^ut et contr\^ole des solutions des \'equations
  hyperboliques.
\newblock {\em Asymptotic Anal.}, 10:95--115, 1995.

\bibitem[Rus71a]{russell1971boundarya}
David~L Russell.
\newblock Boundary value control of the higher-dimensional wave equation.
\newblock {\em SIAM Journal on Control}, 9(1):29--42, 1971.

\bibitem[Rus71b]{russell1971boundaryb}
David~L Russell.
\newblock Boundary value control theory of the higher-dimensional wave
  equation, part ii.
\newblock {\em SIAM Journal on Control}, 9(3):401--419, 1971.

\bibitem[RZ98]{RZ:98}
Luc Robbiano and Claude Zuily.
\newblock Uniqueness in the {C}auchy problem for operators with partially
  holomorphic coefficients.
\newblock {\em Invent. Math.}, 131(3):493--539, 1998.

\bibitem[SU98]{stefanov1998stability}
Plamen Stefanov and Gunther Uhlmann.
\newblock Stability estimates for the hyperbolic {D}irichlet to {N}eumann map
  in anisotropic media.
\newblock {\em Journal of Functional Analysis}, 154(2):330--358, 1998.

\bibitem[SU05]{stefanov2005stable}
Plamen Stefanov and Gunther Uhlmann.
\newblock Stable determination of generic simple metrics from the hyperbolic
  {D}irichlet-to-{N}eumann map.
\newblock {\em International Mathematics Research Notices},
  2005(17):1047--1061, 2005.

\bibitem[Sun90]{sun1990continuous}
Ziqi Sun.
\newblock On continuous dependence for an inverse initial boundary value
  problem for the wave equation.
\newblock {\em Journal of Mathematical Analysis and Applications},
  150(1):188--204, 1990.

\bibitem[Tat95]{Tataru:95}
Daniel Tataru.
\newblock Unique continuation for solutions to {PDE}'s; between {H}\"ormander's
  theorem and {H}olmgren's theorem.
\newblock {\em Comm. Partial Differential Equations}, 20(5-6):855--884, 1995.

\bibitem[Tat99]{tataru1999unique}
Daniel Tataru.
\newblock Unique continuation for operators with partially analytic
  coefficients.
\newblock {\em Journal de math{\'e}matiques pures et appliqu{\'e}es},
  78(5):505--521, 1999.

\end{thebibliography}

\end{document}